\documentclass[12pt]{amsart}

\usepackage{amssymb}
\usepackage{amsmath}
\usepackage{mathdots}
\usepackage{mathtools}
\usepackage{amsbsy}
\usepackage{amscd}
\usepackage{amsthm}
\usepackage{mathrsfs}
\usepackage{verbatim}
\usepackage[colorlinks]{hyperref}
\usepackage[english]{babel}
\usepackage{url}
\usepackage{todonotes}
\usepackage{tikz}

\usetikzlibrary{matrix}
\usepackage[shortlabels]{enumitem}

\usepackage{algorithm}
\usepackage{algpseudocode}

\usepackage{fullpage}

\def\blue#1{{\bf \color{blue} #1}}

\theoremstyle{definition}
\newtheorem{Theorem}{Theorem}[section]

\newtheorem{Example}[Theorem]{Example}
\newtheorem{Remark}[Theorem]{Remark}
\newtheorem{Definition}[Theorem]{Definition}

\newtheorem{Convention}[Theorem]{Convention}

\def\blue#1{{\bf \color{blue} #1}}

\begin{document}

\title{Schubert polynomial analogues for degenerate involutions}

\author{Michael Joyce}
\address{Tulane University}
\email{mjoyce3@tulane.edu}



\date{\today}

\maketitle

\begin{abstract}
We survey the recent study of involution Schubert polynomials and a modest generalization that we call degenerate involution Schubert polynomials. We cite several conditions when (degenerate) involution Schubert polynomials have simple factorization formulae. Such polynomials can be computed by traversing through chains in certain weak order posets, and we provide explicit descriptions of such chains in weak order for involutions and degenerate involutions. As an application, we give several examples of how certain multiplicity-free sums of Schubert polynomials factor completely into very simple linear factors.
\end{abstract}

\section{Introduction}

The study of the complete flag variety $\mathcal{F}$, the algebraic variety that parameterizes the complete flags of an $n$-dimensional (complex) vector space, leads to many rich links between geometry, algebra, representation theory and combinatorics. The general linear group $G = GL_n(\mathbb{C})$ acts transitively on $\mathcal{F}$ and of particular interest are actions of subgroups $H \subset G$ that act with finitely many orbits in $\mathcal{F}$; such subgroups are called spherical.

In this survey, we focus on the case when $H = O_n(\mathbb{C})$, the orthogonal group, and on variations $H_{\mu}$ associated to compositions $\mu$ of $n$. When $\mu$ is the composition with a single part equal to $n$, $H_{\mu} = H = O_n(\mathbb{C})$, while the composition $\mu$ with $n$ parts equal to $1$ corresponds to $H_{\mu} = B$, the Borel subgroup of $G$ consisting of upper triangular matrices. In the general case when $\mu = (\mu_1, \mu_2, \dots, \mu_k)$, $H_{\mu}$ is a semidirect product of a Levi factor $L_{\mu} = O_{\mu_1} \times \dots \times O_{\mu_k}$ (embedded into $G$ diagonally) and the unipotent radical of a parabolic subgroup $P_{\mu} \subset G$. The $H$-orbits on $\mathcal{F}$ are parameterized by involutions, while the $H_{\mu}$-orbits on $\mathcal{F}$ are parameterized by combinatorial objects called $\mu$-involutions. We refer to a $\mu$-involution for some composition $\mu$ of $n$ as a degenerate involution of rank $n$.

The $B$-orbits on $\mathcal{F}$ are the classical Schubert cells of $\mathcal{F}$ and are parameterized by permutations in $S_n$, the symmetric group on $n$ letters. For $w \in S_n$, the Schubert polynomial $\mathfrak{S}_w$ is a natural representative of the class of the closure of the Schubert cell associated to $w$ in $H^*(\mathcal{F})$, the integral cohomology ring of $\mathcal{F}$. Letting $\mathcal{I}_n$ denote the set of involutions in $S_n$, for $\pi \in \mathcal{I}_n$, the involution Schubert polynomial $\hat{\mathfrak{S}}_{\pi}$ is a natural representative of the class of the closure of the $H$-orbit of $\mathcal{F}$ associated to $\pi$ in $H^*(\mathcal{F})$. Likewise, for a fixed composition $\mu$ of $n$, we let $\mathcal{I}_{\mu}$ denote the set of $\mu$-involutions and for $\pi \in \mathcal{I}_{\mu}$, the $\mu$-involution Schubert polynomial $\hat{\mathfrak{S}}^{\mu}_{\pi}$ is a natural representative of the class of the closure of the $H_{\mu}$-orbit of $\mathcal{F}$ associated to $\pi$ in $H^*(\mathcal{F})$.

Let us give a brief history of the study of involution Schubert polynomials and degenerate involution Schubert polynomials. The first detailed investigation of $H$-orbits on $\mathcal{F}$ was carried out in the more general setting of symmetric varieties by Richardson and Springer \cite{RS90, RS94}. One of their key results is that the inclusion order of $H$-orbit closures in $\mathcal{F}$ is given by the restriction of Bruhat order on $S_n$ to $\mathcal{I}_n$ \cite{RS93}. In a combinatorial framework, Can and the author \cite{CJ13} decomposed the (degenerate) involution Schubert polynomial associated to the longest permutation (which can be viewed as a $\mu$-involution for any composition $\mu$) as a sum of ordinary Schubert polynomials, by studying maximal chains in the associated weak order poset, using a result of Brion \cite{Brion98}. Can, Wyser and the author \cite{CJW16} then decomposed an arbitrary involution Schubert polynomial as a multiplicity-free sum of ordinary Schubert polynomials. Hamaker, Marberg and Pawlowski \cite{HMP15} gave the first detailed account of involution Schubert polynomials, creating a uniform combinatorial language for the study and connecting the combinatorics explicitly to the geometry of $H$-orbit closures in $\mathcal{F}$. Hamaker, Marberg and Pawlowski \cite{HMP17} then developed the theory of involution words, describing the maximal chains for any interval in the weak order poset associated to $H$-orbits on $\mathcal{F}$ and providing evidence for their explicit conjecture that such chains are closely linked to ordinary Bruhat order on $S_n$. Can, Wyser and the author \cite{CJW18} return to the study of the degenerate involution Schubert polynomial associated to the longest permutation viewed as a $\mu$-involution and use geometric considerations to show that certain multiplicity-free sums of ordinary Schubert polynomials have very simple factorizations. Hamaker, Marberg and Pawlowski \cite{HMP16} have given a transition formula for involution Schubert polynomials, generalizing the transition formula for ordinary Schubert polynomials established by Lascoux and Sch{\"u}tzenberger \cite{LS85a}. Hamaker, Marberg and Pawlowski \cite{HMP17a} have also initiated a study of involution Stanley symmetric functions, a natural limit of involution Schubert polynomials, and shown that they can be expanded positively in the Schur $P$-basis; they have also established a similar result for fixed-point-free involution Stanley symmetric functions \cite{HMP17b}.

We now describe the contents of this survey. In Section \ref{sec:notation}, we establish our notation and conventions. Then, in Section \ref{sec:K-Schubert}, we explain how to define $K$-Schubert polynomials for a spherical variety $K$ and give three important characterizations in the case $K = H_{\mu}$. We then discuss factorization results for (degenerate) involution Schubert polynomials in Section \ref{sec:factorization}. In Section \ref{sec:chains}, we describe chains in the weak order for $\mu$-involutions. Finally, in Section \ref{sec:identities}, we apply our results to give new identities expressing multiplicity-free sums of ordinary Schubert polynomials as a product of simple linear factors.

\subsection*{Acknowledgements}\label{subsec:acknowledgements}

The author is grateful to Mahir Can for many helpful discussions. The author thanks the referee for their thorough reading of the paper and their helpful suggestions for improvement.

\section{Notation and Conventions}\label{sec:notation}

In this paper, $G$ denotes a connected, reductive algebraic group over $\mathbb{C}$ (though our results apply for any ground field over an algebraically closed field of characteristic $\neq 2$). We let $B$ denote a Borel subgroup of $G$ and $K$ a \emph{spherical subgroup} of $G$, i.e. $K$ contains a dense orbit on $G / B$ under the left multiplication action.  Most of our results are specialized to the case where $G = GL_n(\mathbb{C})$; in that case, we let $B$ be the Borel subgroup of upper triangular matrices for definiteness. In this case, $G/B$ is isomorphic to the complete flag variety $\mathcal{F}$.

Our main choices of spherical subgroup $K$ will be $H = O_n(\mathbb{C})$ and a subgroup $H_{\mu}$ associated to a composition $\mu$ of $n$ that we now describe. Recall that a composition $\mu$ is a sequence of positive integers $(\mu_1, \dots, \mu_k)$ such that $\mu_1 + \cdots + \mu_k = n$. The subgroup $H_{\mu} \subseteq GL_n(\mathbb{C})$ is the semidirect product $L_{\mu} \ltimes R_{\mu}$, where $$L_{\mu} = O_{\mu_1}(\mathbb{C}) \times \cdots \times O_{\mu_k}(\mathbb{C}),$$ embedded diagonally in $GL_n(\mathbb{C})$, and $R_{\mu}$ is the unipotent radical of the parabolic subgroup $P_{\mu}$ containing $B$ whose associated Levi factor is $GL_{\mu_1}(\mathbb{C}) \times \cdots \times GL_{\mu_k}(\mathbb{C})$.

We will make use of some elementary combinatorics of the symmetric group $S_n$. The simple transpositions of $S_n$ are denoted $s_i$, $i = 1, \dots, n-1$, where $s_i$ is the permutation that interchanges $i$ and $i+1$ and fixes all other numbers. Given a permutation $w \in S_n$, a \emph{reduced decomposition} of $w$ is a sequence $(s_{i_1}, \dots, s_{i_{\ell}})$ of simple transpositions such that $w = s_{i_1} \cdots s_{i_{\ell}}$ with $\ell$ minimal. The number $\ell = \ell(w)$ is the \emph{length} of $w$ and is equal to the number of pairs $(i,j)$ with $1 \leq i < j \leq n$ such that $w(i) > w(j)$. As a slight abuse of notation, we identify a reduced decomposition with the corresponding product $s_{i_1} \cdots s_{i_{\ell}}$. The longest permutation of $S_n$ is $w_0$ defined by $w_0(i) = n + 1 - i$.

The group $S_n$ acts on $f \in \mathbb{Z}[x_1, \dots, x_n]$ as follows. For $1 \leq i \leq n - 1$, let $s_i \in S_n$ denote the simple transposition that interchanges $i$ and $i+1$ while fixing all other values. Then $s_i \cdot f$ is the polynomial obtained by interchanging the variables $x_i$ and $x_{i+1}$. Note that $f - s_i \cdot f$ is always divisible by $x_i - x_{i+1}$. Thus, we can define divided difference operators $\partial_i : \mathbb{Z}[x_1, \dots, x_n] \rightarrow \mathbb{Z}[x_1, \dots, x_n]$ by
$$\partial_i(f) = \frac{f - s_i \cdot f}{x_i - x_{i+1}}.$$
It can be easily seen that if $I$ denotes the ideal in $\mathbb{Z}[x_1,\dots,x_n]$ generated by all homogeneous symmetric polynomials in $x_1,\dots,x_n$ of positive degree, then the divided difference operator descends to an operator $\partial_i : \mathbb{Z}[x_1,\dots,x_n]/I \rightarrow \mathbb{Z}[x_1,\dots,x_n]/I$, which by abuse of notation we also refer to as $\partial_i$.

We now recall the definition of the Richardson-Springer monoid \cite{RS90} of the symmetric group $S_n$. Recall that $S_n$ is generated by the $n-1$ simple transpositions $s_i$, $1 \leq i \leq n-1$. The Richardson-Springer monoid, denoted $M(S_n)$, is generated by the elements $m(s_1), \dots, m(s_{n-1})$ subject to the relations $m(s_i)^2 = m(s_i)$, $m(s_i) m(s_j) = m(s_j) m(s_i)$ whenever $|i - j| > 1$ and $m(s_i) m(s_{i+1}) m(s_i) = m(s_{i+1}) m(s_i) m(s_{i+1})$ for $1 \leq i \leq n-2$. Then $M(S_n)$ is a finite monoid and every element of $M(S_n)$ has the form $m(w)$ for $w \in S_n$, where $m(w) := m(s_{i_1}) \cdots m(s_{i_{\ell}})$ for any reduced decomposition $s_{i_1} \cdots s_{i_{\ell}}$ of $w$.

An \emph{involution} is a permutation $w$ such that $w^2$ is the identity permutation. In particular, we consider the identity permutation itself to be an involution. We let $\mathcal{I}_n$ denote the set of all involutions in $S_n$. To define degenerate involutions, we need a convention for interpreting certain strings as permutations. Recall that $[n] = \{1, 2, \dots, n\}$.

\begin{Convention}\label{conv:string perm}
Given any string $\alpha$ containing each element of an alphabet $\mathbb{A} \subset [n]$ in exactly one position, we interpret $\alpha$ as the one-line notation of a permutation of $\mathbb{A}$, where we order $\mathbb{A}$ in increasing order. For example, the string $\alpha = [5264]$ is interpreted as the permutation of the alphabet $\mathbb{A} = \{2,4,5,6\}$ given by $2 \mapsto 5, 4 \mapsto 2, 5 \mapsto 6, 6 \mapsto 4$.
\end{Convention}

Let $\mu = (\mu_1, \dots, \mu_k)$ be a composition of $n$. Then a \emph{$\mu$-involution} is a permutation in $S_n$ such that, when the one-line notation is partitioned into strings of length $\mu_1, \dots, \mu_k$, when each string is viewed as permutation of its alphabet according to Convention \ref{conv:string perm}, the corresponding permutation is an involution. We let $\mathcal{I}_{\mu}$ denote the set of all $\mu$-involutions. A \emph{degenerate involution} (of rank $n$) is a $\mu$-involution for some composition $\mu$ (of $n$).

We adopt a notational convention that $\tau$, $\tau'$, etc. denote involutions, while $\pi$, $\pi'$, etc. denote $\mu$-involutions.

Finally, we introduce a standard notation from the theory of posets. If $X$ is any partially ordered set and $a, b \in X$, then $[a,b] := \{ x \in X : a \leq x \leq b \}$.

\section{$K$-Schubert Polynomials}\label{sec:K-Schubert}

Throughout this section, we only consider the group $G = GL_n(\mathbb{C})$.

The integral singular cohomology ring of $\mathcal{F}$ can be viewed from two complementary perspectives. First, an additive basis for $H^*(\mathcal{F})$ is given by the \emph{Schubert classes}, which are the classes of the \emph{Schubert varieties}. The Schubert varieties are the closures of the orbits of the $B$-action on $\mathcal{F}$, and are parameterized by the permutations $w \in S_n$; the corresponding Schubert variety is denoted by $X_w$ and its Schubert class by $\sigma_w \in H^*(\mathcal{F})$. Schubert varieties can also be described concretely by rank conditions (c.f. \cite{Fulton, Manivel}).

Second, there is the Borel presentation which describes the ring structure, $H^*(\mathcal{F}) \cong \mathbb{Z}[x_1, \dots, x_n] / I$, where $I$ is the ideal generated by all homogeneous symmetric polynomials in $x_1, \dots, x_n$ of positive degree, or equivalently, $I$ is generated by the elementary symmetric polynomials in $x_1, \dots, x_n$. Under this isomorphism, $x_i$ is mapped to the Chern class of the line bundle $L_i / L_{i-1}$ constructed from the tautological sequence of bundles $$ 0= L_0 \subset L_1 \subset \dots \subset L_{n-1} \subset L_n = \mathcal{F} \times \mathbb{C}^n.$$ Here the fiber of $L_i$ over the flag $F \in \mathcal{F}$ is $V_i$ if $F$ is given by $0 = V_0 \subset V_1 \subset \dots \subset V_{n-1} \subset V_n = \mathbb{C}^n$.

A natural question emerges: how are Schubert classes represented in the Borel presentation? Of course, since $H^*(\mathcal{F})$ is a quotient of $\mathbb{Z}[x_1, \dots, x_n]$, there are many different choices of representatives for $\sigma_w$. Lascoux and Sch{\"u}tzenberger found a choice of natural representatives with rich combinatorial structure \cite{LS82a}, which they called \emph{Schubert polynomials}. The Schubert polynomial $\mathfrak{S}_w$, which represents the Schubert class $\sigma_w$, has several characterizations. First, Schubert polynomials can be defined recursively. The Schubert polynomial for the longest permutation $w_0$ is given by $\mathfrak{S}_{w_0} := x_1^{n-1} x_2^{n-2} \cdots x_{n-1}$. Then, for any $w \in S_n$, $\mathfrak{S}_w := \partial_i(\mathfrak{S}_{ws_i})$ for any $i$ such that $\ell(ws_i) = \ell(w) + 1$.

Second, $\mathfrak{S}_w$ can be characterized as the unique polynomial in the $\mathbb{Z}$-span of the Artin basis, $\{ x_1^{a_1} \cdots x_n^{a_n} : 0 \leq a_i \leq n - i \}$, that represents $\sigma_w$ in $\mathbb{Z}[x_1, \dots, x_n] / I$. We adapt the latter point of view in defining the $K$-Schubert polynomial of a $K$-orbit closure $Y$ for any spherical subgroup $K$. There is one technical complication, due to the fact that the greatest common divisor of the coefficients of the polynomial representing a $K$-orbit closure may be greater than $1$. In fact, the gcd will always be of the form $2^{\kappa(Y)}$ for some non-negative integer $\kappa(Y)$ \cite{Brion01}. So we will define the $K$-Schubert polynomial to be the representative polynomial in the Artin basis divided by this common factor.

\begin{Definition}\label{def:K-schub}
Let $K$ be a spherical subgroup of $GL_n(\mathbb{C})$ and let $Y$ be a $K$-orbit closure in $\mathcal{F}$. The \emph{$K$-Schubert polynomial} of $Y$, denoted $\mathfrak{S}_Y$, is the unique polynomial in the $\mathbb{Z}$-span of $\{ x_1^{a_1} \cdots x_n^{a_n} : 0 \leq a_i \leq n - i \}$ such that $2^{\kappa(Y)} \mathfrak{S}_Y$ represents the class of $Y$ in $H^*(\mathcal{F})$. If $K = H$ is the orthogonal group and $Y$ is the closure of $H \tau B$ for some involution $\tau$, then we write $\hat{\mathfrak{S}}_{\tau}$ for $\mathfrak{S}_Y$. More generally, if $K = H_{\mu}$ and $Y$ is the $K$-orbit closure associated to the degenerate involution $\pi$, then we write $\hat{\mathfrak{S}}_{\pi}$ for $\mathfrak{S}_Y$.
\end{Definition}

While Definition \ref{def:K-schub} connects $K$-Schubert polynomials to geometry, it does not give an explicit algebraic or combinatorial description of them. In the cases of interest, we can give a recursive description using divided difference operators, analogous to the case of ordinary Schubert polynomials.

If $K = H  = O_n(\mathbb{C})$, then there is a unique closed $K$-orbit on $\mathcal{F}$, parameterized by the longest permutation $\tau_0$ viewed as an involution. Then as a corollary to Theorem \ref{thm:dominant involution factorization}, we have that $$ \hat{\mathfrak{S}}_{\tau_0} = x_1 \cdots x_{\lfloor{n/2}\rfloor} \prod_{0 < i < j \leq n - i} (x_i + x_j).$$ When $K = H_{\mu}$, there is a unique closed $K$-orbit on $\mathcal{F}$, parameterized by the longest permutation $\pi_0$ viewed as a $\mu$-involution. Then Theorem \ref{thm:mu-involution factorization} gives a similar factorization formula for $\hat{\mathfrak{S}}_{\pi_0}$, and we defer the result until Section \ref{sec:factorization} where the necessary notation is introduced.

Then we can compute $\hat{\mathfrak{S}}_{\tau}$ (resp., $\hat{\mathfrak{S}}_{\pi}$) for an involution $\tau$ (resp., $\mu$-involution $\pi$) recursively via $\hat{\mathfrak{S}}_{\tau} = \partial_i (\hat{\mathfrak{S}}_{\tau'})$ (resp., $\hat{\mathfrak{S}}^{\mu}_{\pi} = \partial_i (\hat{\mathfrak{S}}^{\mu}_{\pi'})$) when $m(s_i) \cdot \tau = \tau'$ (resp., $m(s_i) \cdot \pi = \pi'$), where the latter notation refers to the action of the Richardson-Springer monoid which is described explicitly in Section \ref{sec:chains}.

There is another approach to understanding $K$-Schubert polynomials, and that is to express them in the basis of ordinary Schubert polynomials. For a $K$-orbit closure  $Y$, we may write
$$\hat{\mathfrak{S}}_Y = \sum_{w \in S_n} c_{Y,w} \mathfrak{S}_w,$$
with $c_{Y,w} \in \mathbb{Z}$. In our cases of interest, we will see in Section \ref{sec:identities} that each $c_{Y,w}$ is a non-negative integer (this is true in general for geometric reasons \cite{Brion98}) and that in fact $c_{Y,w}$ is equal to $0$ or $1$.

\section{Factorization Results for Involution Schubert Polynomials}\label{sec:factorization}

\subsection{Permutations}

We recall the definition of the \emph{Rothe diagram} of a permutation $w \in S_n$. It is defined to be the set
$$ D(w) := \{ (i,j) \in [n] \times [n] : j < w(i) \text{ and } i < w^{-1}(j) \}. $$
More concretely, the diagram can be obtained as follows. Begin with the full set $[n] \times [n]$ and then eliminate every entry $(i, w(i))$ as well as every entry directly to the right or directly below that entry. The length of $w$ is equal to the cardinality of $D(w)$. The \emph{code} of $w$ is the sequence $c(w) = (c_1(w), \dots, c_{n}(w))$ where $c_i(w)$ is equal to the number of $j \in [n]$ such that $(i,j) \in D(w)$. We also recall that $D(w^{-1})$ is the transpose of $D(w)$.

We next recall the notion of dominant permutations. We begin by recalling a well-known theorem.

\begin{Theorem}\label{thm:domiant conds}
The following are equivalent for a permutation $w \in S_n$:
\begin{enumerate}
\item The diagram of $w$ is the diagram of a partition, i.e a left-arrayed collection of rows of weakly decreasing length.
\item The code of $w$ is a partition, i.e. $c_1(w) \geq c_2(w) \geq \cdots \geq c_n(w)$.
\item The permutation $w$ is $132$-avoiding, i.e. there is no $i < j < k$ with $w(i) < w(k) < w(j)$.
\end{enumerate}
\end{Theorem}

\begin{Definition}
A permutation satisfying any of the equivalent conditions of Theorem \ref{thm:domiant conds} is said to be \emph{dominant}.
\end{Definition}

\begin{Theorem}\label{thm:dominant factorization}
If $w \in S_n$ is a dominant permutation, then
$$ \mathfrak{S}_w = \prod_{(i,j) \in D(w)} x_i = \prod_{i = 1}^n x_i^{c_i}. $$
\end{Theorem}

\begin{Example}
The permutation $w = [6435721] \in S_7$ is dominant, with code $c(w) = (5,3,2,2,2,1,0)$. Thus, $\mathfrak{S}_w = x_1^5 x_2^3 x_3^2 x_4^2 x_5^2 x_6$.
\end{Example}


\subsection{Involutions}

Following \cite{HMP15}, we define the \emph{involution diagram} of $\tau \in \mathcal{I}_n$. It is important to note that the involution diagram of $\tau \in \mathcal{I}_n$ is different than its diagram when viewed as a permutation in $S_n$. We have
$$ \hat{D}(\tau) := \{ (i,j) \in [n] \times [n] : j < \tau(i) \text{ and } i < \tau(j) \text{ and } i \leq j \}. $$
Thus, $\hat{D}(\tau)$ is equal to the the lower left half of $D(\tau)$ including the diagonal. The \emph{involution length} of $\tau$, denoted $\hat{\ell}(\tau)$, is equal to the cardinality of $\hat{D}(\tau)$. The \emph{involution code} of $\tau$ is the sequence $\hat{c}(\tau) = (\hat{c}_1(\tau), \dots \hat{c}_n(\tau))$ where $\hat{c}_i(\tau)$ is equal to the number of $j \in [n]$ such that $(i,j) \in \hat{D}(\tau)$.

We introduce some definitions needed to state our next result. Let $\kappa(\tau)$ denote the number of disjoint $2$-cycles of $\tau \in \mathcal{I}_n$. Let $\hat{D}_1(\tau) = \{ (i,j) \in \hat{D}(\tau) : i = j \}$. We note that $(a,a) \in \hat{D}_1(\tau)$ if and only if $\tau$ contains a $2$-cycle $(a,b)$ with $a < b$. In particular, $\kappa(\tau) = \# \hat{D}_1(\tau)$. Let $\hat{D}_2(\tau) = \{ (i,j) \in \hat{D}(\tau) : i < j \}$. An involution $\tau \in \mathcal{I}_n$ is said to be dominant if it is dominant as a permutation.

\begin{Theorem}\cite[Theorem 3.26]{HMP15}\label{thm:dominant involution factorization}
If $\tau \in \mathcal{I}_n$ is a dominant involution, then
$$ \hat{\mathfrak{S}}_{\tau} = \prod_{(i,i) \in \hat{D}_1(\tau)} x_i \prod_{(i,j) \in \hat{D}_2(\tau)} (x_i + x_j) = \frac{1}{2^{\kappa(\tau)}} \prod_{(i,j) \in \hat{D}(\tau)} (x_i + x_j). $$
\end{Theorem}

\begin{Example}
The involution $\tau = (1,6)(2,5)(3,7) \in \mathcal{I}_7$ is dominant and $$\hat{\mathfrak{S}}_{\tau} = x_1 x_2 x_3 (x_1 + x_2) (x_1 + x_3) (x_1 + x_4) (x_1 + x_5) (x_2 + x_3) (x_2 + x_4)(x_3 + x_4).$$
\end{Example}

\subsection{Degenerate Involutions}

There is not yet any general theory of diagrams and dominance for degenerate involutions. However, the degenerate involution Schubert polynomial associated to the longest permutation, viewed as a $\mu$-involution for any composition $\mu$, has a very simple factorization.

Let $\mu = (\mu_1, \dots, \mu_k)$ be a composition of $n$ and let $\nu_i = \mu_1 + \cdots + \mu_i$, with the convention $\nu_0 = 0$. Let $\tau_n \in \mathcal{I}_n$ denote the longest permutation on $n$ letters, viewed as an involution. We define the degenerate involution diagram associated to the longest permutation $\pi_{\mu} \in \mathcal{I}_{\mu}$, viewed as a $\mu$-involution. Let
$$\hat{D}^{\mu}_{0} = \{ (i,j) \in [n] \times [n] : \nu_a + 1 \leq i \leq \nu_{a+1} \text{ and } \nu_b + 1 \leq j \leq \nu_{b+1} \text{ with } a < b \},$$
$$\hat{D}^{\mu}_{1} = \{ (i,j) \in [n] \times [n] : \nu_a + 1 \leq i, j \leq \nu_{a+1} \text{ and } (i - \nu_a, j - \nu_a) \in \hat{D}_1(\pi_{\mu_a}) \},$$
$$\hat{D}^{\mu}_{2} = \{ (i,j) \in [n] \times [n] : \nu_a + 1 \leq i, j \leq \nu_{a+1} \text{ and } (i - \nu_a, j - \nu_a) \in \hat{D}_2(\pi_{\mu_a}) \}.$$
Then $\displaystyle \hat{D}^{\mu}(\pi_{\mu}) := \hat{D}^{\mu}_{0} \cup \hat{D}^{\mu}_{1} \cup \hat{D}^{\mu}_{2}$, the union being disjoint.

\begin{Theorem}\cite[Corollary 3.6]{CJW18}\label{thm:mu-involution factorization}
Let $\mu$ be a composition of $n$ and let $\pi_{\mu}$ be the longest permutation, viewed as a $\mu$-involution. Then
$$ \hat{\mathfrak{S}}_{\pi_{\mu}} = \prod_{(i,j) \in \hat{D}^{\mu}_{0}} x_i \prod_{(i,i) \in \hat{D}^{\mu}_{1}} x_i \prod_{(i,j) \in \hat{D}^{\mu}_{2}} (x_i + x_j) = \frac{1}{2^{\hat{\ell}_{\mu}(\pi_{\mu})}} \prod_{(i,j) \in \hat{D}^{\mu}(\pi_{\mu})} (x_i + x_j).$$
\end{Theorem}

The integer $\hat{\ell}_{\mu}(\pi_{\mu})$ is the involution length of $\pi_{\mu}$ and is defined in Section \ref{subsec:deg inv weak}. In the particular case of the longest permutation $\mu$-involution $\pi_{\mu}$, we have $\hat{\ell}^{\mu}(\pi_{\mu}) = \#(\hat{D}^{\mu}_{0} \cup \hat{D}^{\mu}_{1})$.

We remark that the proof of this result in \cite{CJW18} is geometric in nature, relying on the localization theorem in equivariant cohomology. It would be interesting to give a purely combinatorial proof of this result.

\section{Weak Order Chains}\label{sec:chains}

There are well known left and right weak orders on $S_n$ that have a geometric interpretation in terms of minimal parabolic subgroups of $GL_n(\mathbb{C})$ acting on Schubert varieties. In this section, we define the analogues of weak order for involutions and $\mu$-involutions. We will discuss the geometric meaning of these weak orders in Section \ref{sec:identities}.

\subsection{Involutions}

We recall the definition of weak order on involutions. Given $\tau \in \mathcal{I}_n$, we begin by defining an action of the generators of $M(S_n)$ via
$$
m(s_i) \cdot \tau =
\begin{cases}
\tau & \text{if}\ \tau(i+1) < \tau(i) \\
s_i\tau & \text{if}\ \tau(i)=i\ \text{and}\ \tau(i+1)=i+1 \\
s_i \tau s_i & \text{otherwise}
\end{cases}.
$$
It is a straightforward exercise to see that this defines an action of $M(S_n)$. For $\tau, \tau' \in \mathcal{I}_n$, say that $\tau \rightarrow \tau'$ if $\tau' = m(s_i) \cdot \tau$ for some $s_i$. The weak order on $\mathcal{I}_n$ is the transitive closure of the relation $\rightarrow$.

The weak order poset on $\mathcal{I}_n$ is a ranked poset, with $\text{rank}(\tau) = \hat{\ell}(\tau)$ for $\tau \in \mathcal{I}_n$. The poset has both a unique minimal element, the identity involution, and a unique maximal element, the longest permutation, which is an involution because it interchanges $1$ and $n$, $2$ and $n-1$, etc.

We now introduce a central combinatorial problem, to describe maximal chains of intervals in the weak order poset of involutions, using the language of \cite{HMP17}. Let $\tau, \tau' \in \mathcal{I}_n$ and suppose that $\tau \leq \tau'$. An \emph{involution word} from $\tau$ to $\tau'$ is a sequence $(s_{i_1}, \dots, s_{i_k})$ such that $\tau' = m(s_{i_1}) \cdots m(s_{i_k}) \cdot \tau$ and $k = \hat{\ell}(\tau') - \hat{\ell}(\tau)$. Letting $w = s_{i_1} \cdots s_{i_k}$, we may write $\tau' = m(w) \cdot \tau$. An \emph{atom} of $\tau'$ relative to $\tau$ is any $w \in S_n$ such that $\tau' = m(w) \cdot \tau$ and $\ell(w) = \hat{\ell}(\tau') - \hat{\ell}(\tau)$, and the set of all atoms of $\tau'$ relative to $\tau$ is denoted $\mathcal{A}_*(\tau, \tau')$. When $\tau \nleq \tau'$, $\mathcal{A}_*(\tau, \tau') = \emptyset$. When $\tau$ is the identity involution, we define the atoms of $\tau'$ to be $\mathcal{A}(\tau') := \mathcal{A}_*(\text{identity}, \tau')$.

\begin{Remark}\label{rem:chains}
There is a natural surjective map from reduced decompositions of elements in $\mathcal{A}_*(\tau, \tau')$ to maximal chains in the weak order poset of the interval $[\tau, \tau']$ in $\mathcal{I}_n$, in which a reduced decomposition $s_{i_1} \cdots s_{i_k}$ of some $w \in \mathcal{A}_*(\tau, \tau')$ maps to the maximal chain consisting of $\tau, m(s_{i_k}) \cdot \tau, \dots, m(s_{i_1}) \cdots m(s_{i_k}) \cdot \tau = \tau'$. The map may fail to be one-to-one because an edge in the Hasse diagram of $\mathcal{I}_n$ may have more than one $s_j$ labeling it. To get a bijection, one should instead consider the directed graph $\mathscr{G}$ with vertices $\mathcal{I}_n$ and with an edge from $\tau \in \mathcal{I}_n$ to $\tau' \in \mathcal{I}_n$ labeled $j$ if $\tau' = m(s_j) \cdot \tau$. By labeling edges, we allow for the possibility of more than edge from $\tau$ to $\tau'$. Then the reduced decompositions of elements in $\mathcal{A}_*(\tau, \tau')$ correspond bijectively to the maximal paths in the induced subgraph of $\mathscr{G}$ on the vertex set $[\tau, \tau']$.

A similar remark applies to the case of weak order for degenerate involutions considered in Section \ref{subsec:deg inv weak}.
\end{Remark}

For $\tau \in \mathcal{I}_n$, define $\text{Cyc}(\tau) := \{ (i,j) \in [n] \times [n] : j = \tau(i) \text{ and } i \leq j \}$ and $\text{Fix}(\tau) := \{ i \in [n] : \tau(i) = i \}$. The set $\text{Cyc}(\tau)$ describes the $1$- and $2$-cycles of $\tau$, while $\text{Fix}(\tau)$ consists of the fixed points of $\tau$.

We now state a main result describing the atoms of an arbitrary involution.

\begin{Theorem}{\cite[Theorem 2.6]{CJW16}, \cite[Corollary 5.13]{HMP17}}\label{thm:inv w-set}
Let $\tau \in \mathcal{I}_n$. Then $\mathcal{A}(\tau)$ consists of all $w \in S_n$ such that:
\begin{enumerate}
\item If $(i,j) \in \text{Cyc}(\tau)$, then $w(i) \geq w(j)$.
\item If $(i,j) \in \text{Cyc}(\tau)$, then there does not exist $i < k < j$ such that $w(i) > w(k) > w(j)$.
\item If $(i,j), (k,l) \in \text{Cyc}(\tau)$ with $i < k $ and $j < l$, then $w(k) \geq w(l) > w(i) \geq w(j)$.
\end{enumerate}
\end{Theorem}

Hamaker, Marberg and Pawlowski extended this result to describe the relative atoms of an arbitrary pair of involutions.

\begin{Theorem}{\cite[Theorem 5.11]{HMP17}}
Let $\pi, \tau \in \mathcal{I}_n$. Then $\mathcal{A}_*(\tau, \tau')$ consists of all $w \in S_n$ such that for all $(i,j), (k,l) \in \text{Cyc}(\tau')$:
\begin{enumerate}
\item If $w(i) < w(j)$, then $(w(i), w(j) \in \text{Cyc}(\tau)$ and otherwise, $w(i), w(j) \in \text{Fix}(\tau)$.
\item If $i \leq j < k \leq l$, then $w(i) < w(k)$, $w(i) < w(l)$, $w(j) < w(k)$ and $w(j) < w(l)$.
\item If $i < k < j < l$, then $w(i) < w(k)$, $w(i) < w(l)$ and $w(j) < w(l)$.
\item If $i < k < l < j$, then none of the following inequalities holds:
	\begin{enumerate}
	\item $w(j) < w(k) < w(i)$
	\item $w(j) < w(l) < w(i)$
	\item $w(k) < w(i) < w(j) < w(l)$
	\item $w(k) < w(j) \leq w(i) < w(l)$
	\end{enumerate}
\item If $i < k = l < j$, then it is not the case that $w(j) < w(k) = w(l) < w(i)$.
\end{enumerate}
\end{Theorem}

\begin{figure}[htp]
\begin{center}

\begin{tikzpicture}[scale=.47]

\node at (0,0) (a) {$\text{id}$};

\node at (-7.5,5) (b1) {$(1,2)$};
\node at (-2.5,5) (b2) {$(2,3)$};
\node at (2.5,5) (b3) {$(3,4)$};
\node at (7.5,5) (b4) {$(4,5)$};

\node at (-12.5,10) (c1) {$(1,3)$};
\node at (-7.5,10) (c2) {$(1,2)(3,4)$};
\node at (-2.5,10) (c3) {$(1,2)(4,5)$};
\node at (2.5,10) (c4) {$(2,4)$};
\node at (7.5,10) (c5) {$(2,3)(4,5)$};
\node at (12.5,10) (c6) {$(3,5)$};

\node at (-12.5,15) (d1) {$(1,4)$};
\node at (-7.5,15) (d2) {$(1,3)(4,5)$};
\node at (-2.5,15) (d3) {$(1,3)(2,4)$};
\node at (2.5,15) (d4) {$(2,4)(3,5)$};
\node at (7.5,15) (d5) {$(1,2)(3,5)$};
\node at (12.5,15) (d6) {$(2,5)$};

\node at (-5,20) (e1) {$(1,5)$};
\node at (-10,20) (e2) {$(1,4)(2,3)$};
\node at (-0,20) (e3) {$(1,3)(2,5)$};
\node at (10,20) (e4) {$(2,5)(3,4)$};
\node at (5,20) (e5) {$(1,4)(3,5)$};

\node at (-5,25) (f1) {$(1,5)(2,3)$};
\node at (0,25) (f2) {$(1,4)(2,5)$};
\node at (5,25) (f3) {$(1,5)(3,4)$};

\node at (0,30) (g) {$(1,5)(2,4)$};

\node at (-4,2) {$\blue{s_1}$};
\node at (-1.8,2) {$\blue{s_2}$};
\node at (1.7,2) {$\blue{s_3}$};
\node at (4,2) {$\blue{s_4}$};

\node at (-9.5,6) {$\blue{s_2}$};
\node at (-8,6) {$\blue{s_3}$};
\node at (-6.2,5.6) {$\blue{s_4}$};
\node at (-3.7,6) {$\blue{s_1}$};
\node at (-2,6) {$\blue{s_3}$};
\node at (-1,5.2) {$\blue{s_4}$};
\node at (1,5.2) {$\blue{s_1}$};
\node at (2,6) {$\blue{s_2}$};
\node at (3.7,6) {$\blue{s_4}$};
\node at (6.2,5.2) {$\blue{s_1}$};
\node at (8,6) {$\blue{s_2}$};
\node at (9.5,6) {$\blue{s_3}$};

\node at (-13,11) {$\blue{s_3}$};
\node at (-10.6,11) {$\blue{s_4}$};
\node at (-7.1,11) {$\blue{s_2}$};
\node at (-5,10.4) {$\blue{s_4}$};
\node at (-2.5,10.8) {$\blue{s_2}$};
\node at (-0.5,10.4) {$\blue{s_3}$};
\node at (1.5,10.8) {$\blue{s_1}$};
\node at (3.3,10.6) {$\blue{s_4}$};
\node at (5.5,10.4) {$\blue{s_1}$};
\node at (7.2,10.8) {$\blue{s_3}$};
\node at (10.7,11) {$\blue{s_1}$};
\node at (13,11) {$\blue{s_2}$};

\node at (-12.7,16) {$\blue{s_2}$};
\node at (-10.5,15.6) {$\blue{s_4}$};
\node at (-6.2,15.6) {$\blue{s_3}$};
\node at (-3.7,15.8) {$\blue{\{s_1,s_3\}}$};
\node at (-1.5,15.8) {$\blue{s_4}$};
\node at (2.5,16) {$\blue{s_1}$};
\node at (4.2,15.9) {$\blue{\{s_2,s_4\}}$};
\node at (7,15.7) {$\blue{s_2}$};
\node at (10.8,16) {$\blue{s_1}$};
\node at (12.6,16) {$\blue{s_3}$};

\node at (-9.8,21) {$\blue{s_4}$};
\node at (-5.8,21) {$\blue{s_2}$};
\node at (-3.2,20.4) {$\blue{s_3}$};
\node at (-1.2,20.8) {$\blue{s_1}$};
\node at (0.4,21) {$\blue{s_3}$};
\node at (3.5,20.8) {$\blue{s_2}$};
\node at (5.4,21) {$\blue{s_4}$};
\node at (9.6,21) {$\blue{s_1}$};

\node at (-4.8,26) {$\blue{s_3}$};
\node at (-0.15,26) {$\blue{\{s_1,s_4\}}$};
\node at (4.8,26) {$\blue{s_2}$};

\draw[-, very thick, double]  (a) to (b1);
\draw[-, very thick, double]  (a) to (b2);
\draw[-, very thick, double] (a) to (b3);
\draw[-, very thick, double] (a) to (b4);

\draw[-, very thick] (b1) to (c1);
\draw[-, very thick, double] (b1) to (c2);
\draw[-, very thick, double] (b1) to (c3);
\draw[-, very thick] (b2) to (c1);
\draw[-, very thick] (b2) to (c4);
\draw[-, very thick, double] (b2) to (c5);
\draw[-, very thick, double] (b3) to (c2);
\draw[-, very thick] (b3) to (c4);
\draw[-, very thick] (b3) to (c6);
\draw[-, very thick, double] (b4) to (c3);
\draw[-, very thick, double] (b4) to (c5);
\draw[-, very thick] (b4) to (c6);

\draw[-, very thick] (c1) to (d1);
\draw[-, very thick, double] (c1) to (d2);
\draw[-, very thick] (c2) to (d3);
\draw[-, very thick] (c2) to (d5);
\draw[-, very thick] (c3) to (d2);
\draw[-, very thick] (c3) to (d5);
\draw[-, very thick] (c4) to (d1);
\draw[-, very thick] (c4) to (d6);
\draw[-, very thick] (c5) to (d2);
\draw[-, very thick] (c5) to (d4);
\draw[-, very thick, double] (c6) to (d5);
\draw[-, very thick] (c6) to (d6);

\draw[-, very thick] (d1) to (e1);
\draw[-, very thick, double] (d1) to (e2);
\draw[-, very thick] (d2) to (e5);
\draw[-, very thick] (d3) to (e2);
\draw[-, very thick] (d3) to (e3);
\draw[-, very thick] (d4) to (e4);
\draw[-, very thick] (d4) to (e5);
\draw[-, very thick] (d5) to (e3);
\draw[-, very thick] (d6) to (e1);
\draw[-, very thick, double] (d6) to (e4);

\draw[-, very thick, double] (e1) to (f1);
\draw[-, very thick, double] (e1) to (f3);
\draw[-, very thick] (e2) to (f1);
\draw[-, very thick] (e3) to (f1);
\draw[-, very thick] (e3) to (f2);
\draw[-, very thick] (e4) to (f3);
\draw[-, very thick] (e5) to (f2);
\draw[-, very thick] (e5) to (f3);

\draw[-, very thick] (f1) to (g);
\draw[-, very thick] (f2) to (g);
\draw[-, very thick] (f3) to (g);

\end{tikzpicture}

\caption{Weak order on $\mathcal{I}_5$.}
\label{fig:inv 5}

\end{center}
\end{figure}
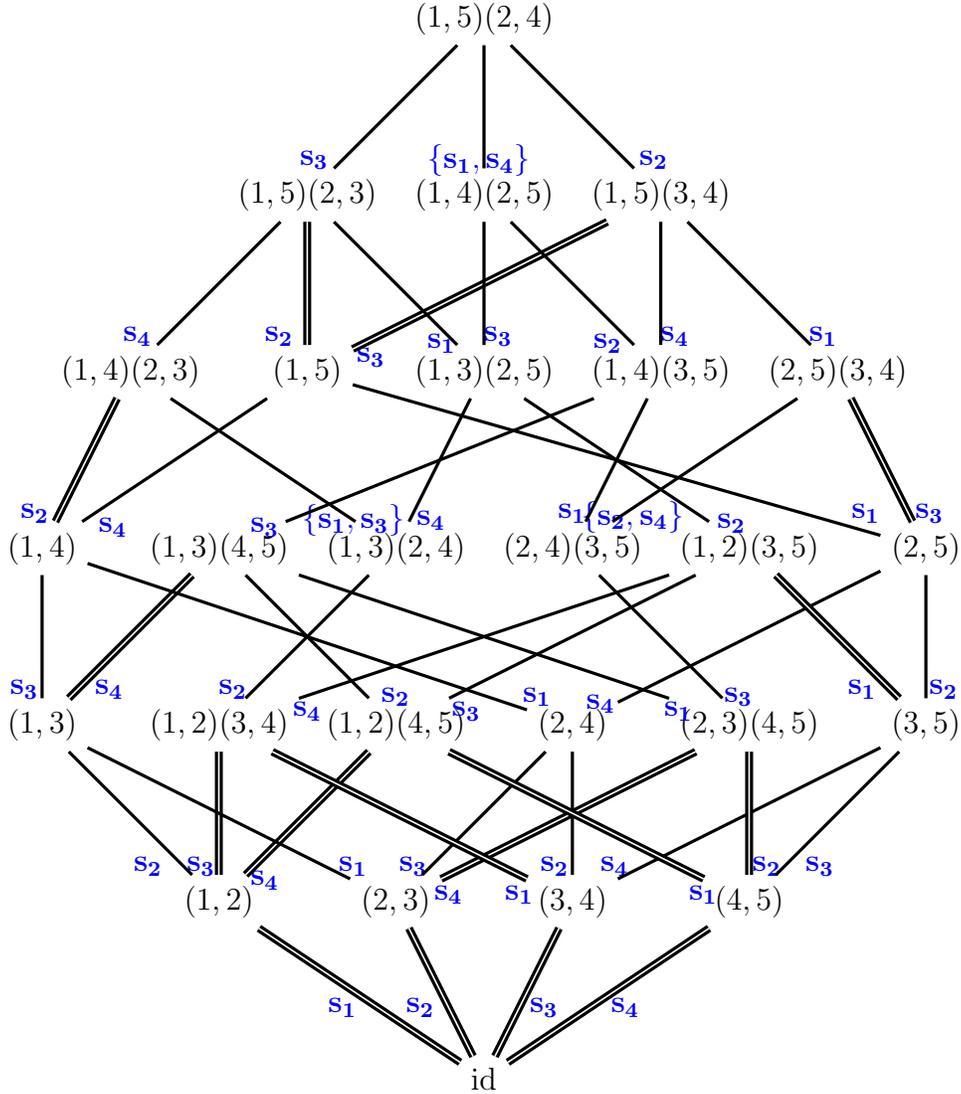

\subsection{Degenerate Involutions}\label{subsec:deg inv weak}

We now recall the definition of weak order on $\mu$-involutions. We again define an action of the generators of $M(S_n)$, but this time the definition is more involved.

\begin{Definition}
Let $\pi = [\alpha_1 | \alpha_2 | \cdots | \alpha_k] \in \mathcal{I}_{\mu}$. Then $s_i \cdot \pi$ is defined according to the following exhaustive list of mutually disjoint cases.
\begin{enumerate}
\item If $i$ occurs after $i+1$ in $\pi$, then $s_i \cdot \pi = \pi$.
\item If $i$ occurs before $i+1$ in $\pi$ and $i$ and $i+1$ occur in different $\mu$-strings of $\pi$, then $s_i \cdot \pi = s_i \pi$, where $s_i \pi$ is the $\mu$-involution obtained from $\pi$ by interchanging the values of $i$ and $i+1$.
\item If $i$ occurs before $i+1$ in $\pi$ and $i$ and $i+1$ occur in the same $\mu$-string $\alpha_j$, then there are two subcases to consider.
    \begin{enumerate}
    \item If $\alpha_j$ fixes both $i$ and $i+1$ (when the string $\alpha_j$ is viewed as a permutation as in Convention \ref{conv:string perm}), then $s_i \cdot \pi = [\alpha_1 | \cdots | s_i \alpha_j | \cdots | \alpha_k]$.
    \item Otherwise, $s_i \cdot \pi = [\alpha_1 | \cdots | s_i \alpha_j s_i | \cdots | \alpha_k]$.
    \end{enumerate}
\end{enumerate}
\end{Definition}

It is a more tedious, but still straightforward, exercise to see that this defines an action of $M(S_n)$. (It also follows from geometry, using arguments from \cite{RS90, Brion01}.) For $\pi, \pi' \in \mathcal{I}_{\mu}$, say that $\pi \rightarrow \pi'$ if $\pi' = m(s_i) \cdot \pi$ for some $s_i$. The weak order on $\mathcal{I}_{\mu}$ is the transitive closure of the relation $\rightarrow$.

The weak order poset on $\mathcal{I}_{\mu}$ is a ranked poset, with $\text{rank}(\pi) = \hat{\ell}_{\mu}(\pi)$ for $\pi \in \mathcal{I}_{\mu}$, where $\hat{\ell}_{\mu}$ is defined as follows. Let $\pi = [\alpha_1 | \cdots | \alpha_k]$ be decomposed into its $\mu$-strings. Then, using Convention \ref{conv:string perm}, each $\alpha_i$ can be viewed as an involution of its alphabet and hence has an involution length $\hat{\ell}(\alpha_i)$. Let $\text{sort}(\pi)$ be the permutation whose one-line notation is obtained by concatenating the increasing rearrangements of each $\mu$-string $\alpha_i$. Then $$\hat{\ell}_{\mu}(\pi) := \sum_{i=1}^k \hat{\ell}(\alpha_i) + \ell(\text{sort}(\pi)).$$

For example, if $\pi = [586 | 21 | 743]$, then $\hat{\ell}(\alpha_1) = 1$, $\hat{\ell}(\alpha_2) = 1$, $\hat{\ell}(\alpha_3) = 2$, while $\text{sort}(\pi) = [56812347]$ and $\ell(\text{sort}(\pi)) = 13$, so $\hat{\ell}_{\mu}(\pi) = 17$.

We now define the notion of atoms for $\mu$-involutions, which is a straightforward generalization of the notion of atoms for involutions. Let $\pi, \pi' \in \mathcal{I}_{\mu}$ and suppose that $\pi \leq \pi'$. An involution word from $\pi$ to $\pi'$ is a sequence $(s_{i_1}, \dots, s_{i_k})$ such that $\pi' = m(s_{i_1}) \cdots m(s_{i_k}) \cdot \pi$ and $k = \hat{\ell}_{\mu}(\pi') - \hat{\ell}_{\mu}(\pi)$. Letting $w = s_{i_1} \cdots s_{i_k}$, we may write $\pi' = m(w) \cdot \pi$. An atom of $\pi'$ relative to $\pi$ is any $w \in S_n$ such that $\pi' = m(w) \cdot \pi$ and $\ell(w) = \hat{\ell}_{\mu}(\pi') - \hat{\ell}_{\mu}(\pi)$, and the set of all atoms of $\pi'$ relative to $\pi$ is denoted $\mathcal{A}_*(\pi, \pi')$. When $\pi \nleq \pi'$, $\mathcal{A}_*(\pi, \pi') = \emptyset$. When $\pi$ is the identity $\mu$-involution, we define the atoms of $\pi'$ to be $\mathcal{A}(\pi') := \mathcal{A}_*(\text{identity}, \pi')$.

Let $w \in S_n$ and fix a composition $\mu = (\mu_1, \dots, \mu_k)$ of $n$. Via its one line notation, we can view $w$ as a string of length $n$. We subdivide $w$ into $k$ strings, where the $i$-th string is $\text{str}_i(w)$ consists of the string of length $\mu_i$ consisting of the elements in positions $\nu_{i} + 1$ through $\nu_{i+1}$. For example, if $w = [37184265]$ and $\mu = (4,1,3)$, then $\text{str}_1(w) = 3718$, $\text{str}_2(w) = 4$ and $\text{str}_3(w) = 265$.

For any integer $n$, let $\tau_n$ denote the longest permutation of $S_n$, viewed as an involution in $\mathcal{I}_n$ and define $\mathcal{A}_n := \mathcal{A}(\tau_n)$. For any composition $\mu$, let $\pi_{\mu}$ denote the longest permutation of $S_n$ viewed as a $\mu$-involution in $\mathcal{I}_{\mu}$ and define $\mathcal{A}_{\mu} := \mathcal{A}(\pi_{\mu})$.

\begin{Theorem}\cite[Proposition 2.5]{CJW18}\label{thm:mu-inv w-set}
The set $\mathcal{A}_{\mu}$ consists of all $w \in S_n$ such that the letters of $\text{str}_i(w)$ are $\nu_{k-i} + 1$ through $\nu_{k+1-i}$ and, viewed as a permutation of its alphabet via Convention \ref{conv:string perm}, $\text{str}_i(w)$ belongs to $\mathcal{A}_{\mu_i}$.
\end{Theorem}

For example, if $\mu = (4,1,3)$, $$\mathcal{A}_{\mu} = \{76854231, 76854312, 78564231, 78564312, 85764231, 85764312 \}.$$

\begin{figure}[htp]
\begin{center}

\begin{tikzpicture}[scale=.57]

\node at (0,0) (a) {$[432|1]$};

\node at (-5,-5) (b1) {$[324|1]$};
\node at (0,-5) (b2) {$[243|1]$};
\node at (5,-5) (b3) {$[431|2]$};

\node at (-10,-10) (c1) {$[314|2]$};
\node at (-4,-10) (c2) {$[234|1]$};
\node at (4,-10) (c3) {$[143|2]$};
\node at (10,-10) (c4) {$[421|3]$};

\node at (-10,-15) (d1) {$[214|3]$};
\node at (-4,-15) (d2) {$[134|2]$};
\node at (4,-15) (d3) {$[142|3]$};
\node at (10,-15) (d4) {$[321|4]$};

\node at (-5,-20) (e1) {$[124|3]$};
\node at (0,-20) (e2) {$[213|4]$};
\node at (5,-20) (e3) {$[132|4]$};

\node at (0,-25) (g) {$[123|4]$};

\node at (-3.2,-2.5) {$\blue{s_3}$};
\node at (3.5,-2.5) {$\blue{s_1}$};
\node at (.5,-2.5) {$\blue{s_2}$};

\node at (-8.6,-8) {$\blue{s_1}$};
\node at (-6.4,-9.3) {$\blue{s_3}$};

\node at (-4.6,-9) {$\blue{s_2}$};
\node at (-2.6,-9) {$\blue{s_3}$};
\node at (2.5,-9) {$\blue{s_1}$};
\node at (8,-8.8) {$\blue{s_2}$};

\node at (-9.5,-13.7) {$\blue{s_2}$};
\node at (-4.5,-13.5) {$\blue{s_1}$};
\node at (-1.5,-14) {$\blue{s_3}$};
\node at (3.5,-13.5) {$\blue{s_2}$};
\node at (6.4,-13.7) {$\blue{s_1}$};
\node at (10.5,-13.7) {$\blue{s_3}$};

\node at (-6.8,-18.8) {$\blue{s_1}$};
\node at (-5.2,-18.8) {$\blue{s_2}$};
\node at (-1.5,-18.8) {$\blue{s_3}$};
\node at (1.5,-18.8) {$\blue{s_2}$};
\node at (4.4,-18.8) {$\blue{s_3}$};
\node at (6.9,-18.8) {$\blue{s_1}$};

\node at (-2.5,-23.1) {$\blue{s_3}$};
\node at (2.5,-23.1) {$\blue{s_2}$};
\node at (.5,-23.1) {$\blue{s_1}$};

\draw[-, very thick]  (a) to (b1);
\draw[-, very thick]  (a) to (b2);
\draw[-, very thick] (a) to (b3);

\draw[-, very thick] (b1) to (c1);
\draw[-, very thick, double] (b1) to (c2);
\draw[-, very thick, double] (b2) to (c2);
\draw[-, very thick] (b2) to (c3);
\draw[-, very thick] (b3) to (c1);
\draw[-, very thick] (b3) to (c4);

\draw[-, very thick] (c1) to (d1);
\draw[-, very thick] (c2) to (d2);
\draw[-, very thick, double] (c3) to (d2);
\draw[-, very thick] (c3) to (d3);
\draw[-, very thick] (c4) to (d3);
\draw[-, very thick] (c4) to (d4);

\draw[-, very thick, double] (d1) to (e1);
\draw[-, very thick] (d1) to (e2);
\draw[-, very thick] (d2) to (e1);
\draw[-, very thick] (d3) to (e3);
\draw[-, very thick] (d4) to (e3);
\draw[-, very thick] (d4) to (e2);

\draw[-, very thick] (e1) to (g);
\draw[-, very thick, double] (e2) to (g);
\draw[-, very thick, double] (e3) to (g);

\end{tikzpicture}

\caption{Weak order on $\mathcal{I}_{3,1}$.}
\label{fig:mu-inv 3,1}

\end{center}
\end{figure}
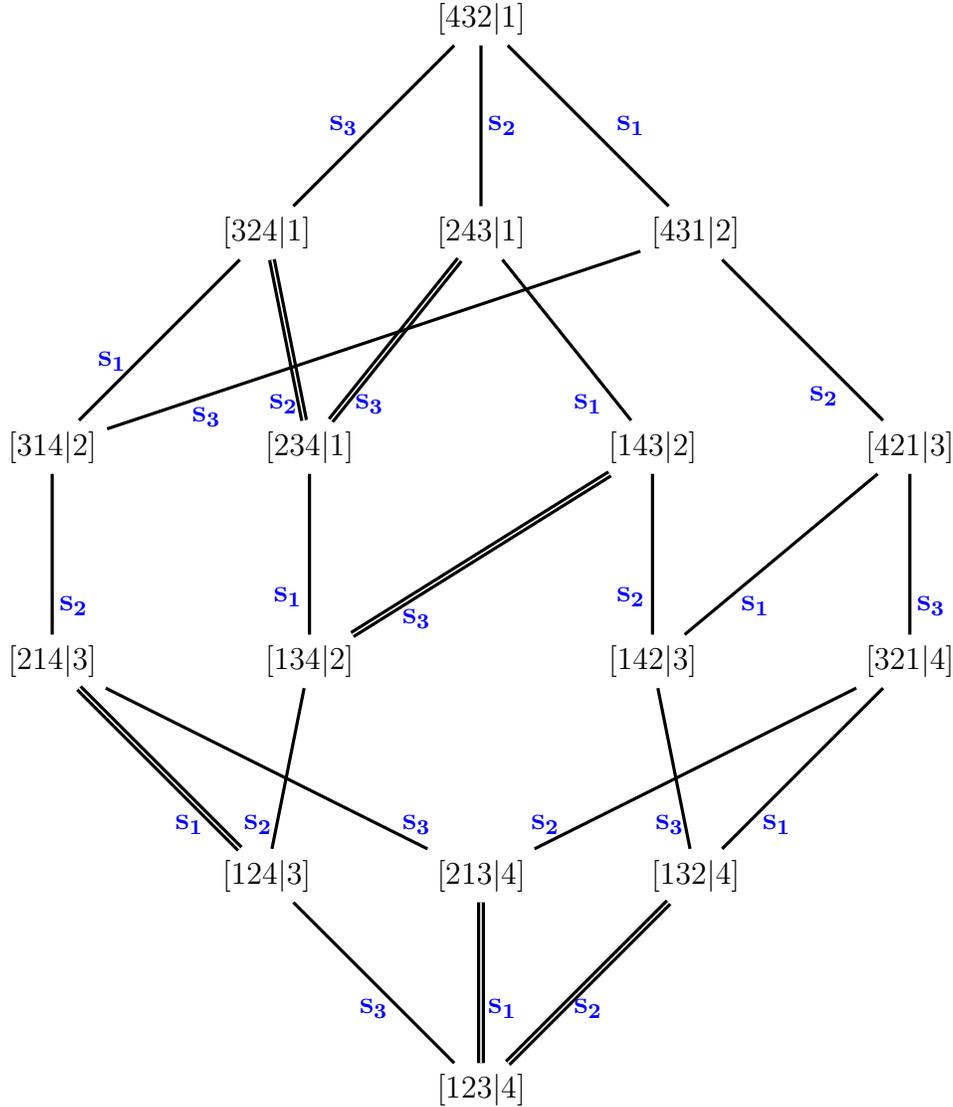

\section{Schubert Polynomial Identities}\label{sec:identities}

We now give results for how certain sums of Schubert polynomials factor completely. The primary tool is a result of Brion which relates the expansion of cohomology classes of $K$-orbit closures for spherical subgroups $K$ in the Schubert basis to chains in a weak order poset associated to $K$. We give a brief description of this result, referring to \cite{Brion98, Brion01} for more details.

To begin, let $G$ be an arbitrary connected reductive group, $B$ a Borel subgroup of $G$ and $K$ a spherical subgroup of $G$. By definition, $K$ has a dense orbit in $G/B$, which we denote $Y_0$ in the sequel. The set $\mathcal{O}_K$ of all $K$-orbits $Y$ in $G/B$ acquires a partial order, called weak order, in which $Y_0$ is the unique maximal element.

The minimal parabolic subgroups of $G$ containing $B$ are denoted $P_s$, where $s$ runs over all of the simple reflections associated to the simple roots of the root system for $(G,B)$. For any $s$, let $p_s : G/B \rightarrow G/P_s$ be the natural projection map. For two distinct $K$-orbits $Y_1$ and $Y_2$, write $Y_1 \rightarrow Y_2$ via $s$, or $Y_1 \xrightarrow{\text{$s$}} Y_2$, if the following occurs: $Y_2$ is the dense orbit of $p_s^{-1}(p_s(\overline{Y_1}))$ for some $s$. The weak order on the set of $K$-orbits in $G/B$ is the transitive closure of the relation $\rightarrow$.

For our applications, we now specialize to $G = GL_n(\mathbb{C})$. In this case, the above construction allows us to define an action of $M(S_n)$ on $\mathcal{O}_K$. In this case, the $s$ parameterizing minimal parabolic subgroups are precisely the simple transpositions $s_i$ of $S_n$ and $G/B \cong \mathcal{F}$. If $Y$ is a $K$-orbit in $\mathcal{F}$, then $m(s_i) \cdot Y$ is defined to be the dense $K$-orbit in $p_{s_i}^{-1}(p_{s_i}(\overline{Y}))$. This action on generators yields a well-defined action of $M(S_n)$ on $\mathcal{O}_K$.

The weak order is intimately related to the theory of $K$-Schubert polynomials. If $Y_1 \xrightarrow{\text{$s_i$}} Y_2$, then $\mathfrak{S}_{Y_2} = \partial_i \mathfrak{S}_{Y_1}$. (In the general setup, one can define more general divided difference operators $\partial_s$, as in \cite{BGG73, Demazure74}, but making a choice of polynomial representative for the Schubert classes is a more subtle problem.)

\begin{Definition}
Let $Y$ be a $K$-orbit on $\mathcal{F}$. Define a set $\mathcal{W}(Y) \subseteq S_n$ as follows. The set $\mathcal{W}(y)$ consists of all permutations $w \in S_n$ that admit a reduced decomposition $w = s_{i_1} \cdots s_{i_\ell}$ such that there are $K$-orbits $Y_1, \dots, Y_{\ell}$ such that
$$Y = Y_{\ell} \xrightarrow{\text{$s_{i_{\ell}}$}} Y_{\ell - 1} \xrightarrow{\text{$s_{i_{\ell-1}}$}} \cdots \xrightarrow{\text{$s_{i_2}$}} Y_1 \xrightarrow{\text{$s_{i_1}$}} Y_0$$
and $\dim Y_{i-1} = \dim Y_{i} + 1$ for $1 \leq \i \leq \ell$. (In particular, the codimension of $Y$ in $\mathcal{F}$ must be $\ell$.)
\end{Definition}

Brion \cite{Brion98} proved an important theorem expressing the class $[\overline{Y}] \in H^*(\mathcal{F})$ positively in the Schubert basis. Reinterpreted in terms of $K$-Schubert polynomials, it says

\begin{Theorem}{\cite[Theorem 1.5(ii)]{Brion98}}\label{thm:w-set}
Let $Y$ be a $K$-orbit closure in $\mathcal{F}$. Then
$$ \mathfrak{S}_Y = \sum_{w \in \mathcal{W}(Y)} \mathfrak{S}_{w},$$
where $w_0$ denotes the longest permutation in $S_n$, i.e. $w_0(i) = n + 1 - i$.
\end{Theorem}

In the case where $K = H = O_n(\mathbb{C})$ (resp., $K = H_{\mu}$), the weak order on involutions (resp. $\mu$-involutions) is the \emph{opposite} of the weak order for the $K$-orbits on $\mathcal{F}$. In particular, if $Y(\tau)$ is the $K$-orbit on $\mathcal{F}$ corresponding to $\tau \in \mathcal{I}_n$, then $w \in \mathcal{W}(Y(\tau))$ if and only if $w^{-1} \in \mathcal{A}(\tau)$. Similarly, if $Y(\pi)$ is the $K$-orbit on $\mathcal{F}$ corresponding to $\pi \in \mathcal{I}_{\mu}$, then $w \in \mathcal{W}(Y(\pi))$ if and only if $w^{-1} \in \mathcal{A}(\pi)$.

\begin{Theorem}\label{thm:schub poly ident inv}
Let $\pi \in \mathcal{I}_n$ be a dominant involution in $S_n$, and let $Y(\tau)$ denote the $K$-orbit of $\tau$ in $\mathcal{F} \cong GL_n(\mathbb{C}) / B$. Then
$$ \sum_{w \in \mathcal{A}(\tau)} \mathfrak{S}_{w^{-1}} = \prod_{(i,i) \in \hat{D}_1(\tau)} x_i \prod_{(i,j) \in \hat{D}_2(\tau)} (x_i + x_j). $$
\end{Theorem}

\begin{proof}
The left hand side is equal to $\mathfrak{S}_{Y(\tau)}$ by Theorem \ref{thm:w-set}, while the right hand side is equal to $\mathfrak{S}_{Y(\tau)}$ by Theorem \ref{thm:mu-involution factorization}.
\end{proof}

\begin{Example}
Let $\tau = (1,5)(2,3) \in \mathcal{I}_5$. Then $\mathcal{A}(\tau) = \{32451, 32514, 35124, 51324\}$ and
$$\mathfrak{S}_{52134} + \mathfrak{S}_{42153} + \mathfrak{S}_{34152} + \mathfrak{S}_{24351} = x_1 x_2 (x_1 + x_2) (x_1 + x_3)(x_1 + x_4).$$
\end{Example}

\begin{Theorem}\label{thm:schub poly ident mu-inv}
Let $\pi \in \mathcal{I}_n$ be the longest permutation in $S_n$ viewed as a $\mu$-involution, and let $Y(\pi)$ denote the corresponding (closed) $K$-orbit of $\pi$ in $\mathcal{F} \cong GL_n(\mathbb{C}) / B$. Then
$$ \sum_{w \in \mathcal{A}_{\mu}} \mathfrak{S}_{w^{-1}} = \prod_{(i,j) \in \hat{D}_0^{\mu}} x_i \prod_{(i,i) \in \hat{D}_1^{\mu}} x_i \prod_{(i,j) \in \hat{D}_2^{\mu}} (x_i + x_j). $$
\end{Theorem}

\begin{proof}
The left hand side is equal to $\mathfrak{S}_{Y(\pi)}$ by Theorem \ref{thm:w-set}, while the right hand side is equal to $\mathfrak{S}_{Y(\pi)}$ by Theorem \ref{thm:dominant involution factorization}.
\end{proof}

\begin{Example}
Let $\mu = (3,1)$ and let $\pi_{\mu} = [432 | 1]$, the longest permutation considered as $\mu$-involution. Then $\mathcal{A}_{\mu} = \{ 4231, 4312 \}$ and
$$\mathfrak{S}_{4231} + \mathfrak{S}_{3421} = x_1^2 x_2 x_3 (x_1 + x_2).$$
\end{Example}

\section{Other Directions}\label{sec:directions}

We close with a list of several open problems in the field.

\begin{enumerate}[label = (\arabic*)]
\item It is possible to define a natural notion of diagrams for $\mu$-involutions. However, there does not appear to be an obvious notion of dominance for $\mu$-involutions. In particular, a $\mu$-involution whose underlying permutation is dominant may not factor into linear factors. Can one define such a notion? In particular, can one find a large class of $\mu$-involutions whose degenerate involution Schubert polynomials factor completely into simple linear factors described in terms of the combinatorics of the associated diagrams?

\item There is an obvious bijection between $K$-orbits on $G/B$ and $B$-orbits on $G/K$. One can transport the weak order structure for $B$-orbits on $G/K$. In general, $G/K$ will not be a complete variety, so one is naturally led to study completions of $G/K$. In many cases, including when $G = GL_n(\mathbb{C})$ and $K = H = O_n(\mathbb{C})$, there is a natural completion to use, called the wonderful compactification of $G/K$ \cite{DP83}. The wonderful compactification of $G/H$ is the classical variety $X$ of complete quadrics, and the $G$-orbits are in bijection with compositions $\mu$ of $n$. Moreover, the stabilizer subgroup of a point in the $G$-orbit associated to $\mu$ is conjugate to the subgroup $H_{\mu}$. Thus, a geometric study of $X$ can unite the various cases considered here. In particular, the $B$-orbits on $X$ are parameterized by the degenerate involutions of rank $n$, and the weak order on $X$ is the disjoint union of the weak orders on $\mu$-involutions for all compositions $\mu$ of $n$.

    The geometry of $X$ is intricate. The cohomology ring of $X$ has been studied directly \cite{DGMP88} and as an example of a complete symmetric variety \cite{BDP90}. Recently, several combinatorial models related to the geometry of $X$ have been introduced \cite{BCJ16}. Still, there is much to be learned about $X$. Are there are any hidden symmetries in $H^*(X)$? Can one give a combinatorial description of the analogue of the Bruhat order, the inclusion order of $B$-orbit closures on $X$?

\item The combinatorics studied here relates to the symmetric subgroup $O_n(\mathbb{C})$ of $GL_n(\mathbb{C})$. There are similar results for the symmetric subgroups $Sp_n(\mathbb{C})$ \cite{CJW16, HMP15, HMP17, WY17} and $GL_p(\mathbb{C}) \times GL_q(\mathbb{C})$, $p + q = n$, \cite{WY14, CJW16} of $GL_n(\mathbb{C})$ . But there are symmetric subgroups associated to other reductive algebraic groups as well. There are four more classical families of symmetric subgroups and 12 exceptional symmetric subgroups \cite[Table 26.3]{Timashev11}, and it would be interesting to extend the combinatorial and geometric results from the `Type A' cases to the other Lie types.
\end{enumerate}


\end{document}